\documentclass[12pt, a4paper, reqno]{amsart}
\usepackage{amscd, amssymb}
\usepackage{amsthm}
\usepackage{tikz}
\usetikzlibrary{arrows,positioning} 
\usepackage{mathrsfs}
\usepackage{bbm}
\usepackage{ stmaryrd }

\usepackage{hyperref}
\hypersetup{
	colorlinks=true,
	linkcolor=blue,
	filecolor=magenta,  
	urlcolor=cyan,
}

 \usepackage[stable]{footmisc}

\usepackage{hyperref}
\usepackage[capitalise]{cleveref}

\def\End{\operatorname{End}}

\def\ker{\operatorname{ker}}

\def\im{\operatorname{Im}}

\def\dim{\operatorname{dim}}

\def\id{\operatorname{id}}

\def\C{\mathbb{C}}

\def\N{\mathbb{N}}

\def\CC{\mathcal{C}}

\def\Ext{\text{Ext}}

\def\sub{\subseteq}

\def\xto{\xrightarrow}

\newtheorem{thm}{Theorem}[section]
\newtheorem{cor}[thm]{Corollary}
\newtheorem{lemma}[thm]{Lemma}
\newtheorem{prop}[thm]{Proposition}

\theoremstyle{definition}

\theoremstyle{remark}
\newtheorem{remark}[thm]{Remark}

\numberwithin{equation}{section}

\usepackage{relsize}
\usepackage{fancyhdr}
\usepackage[all,cmtip]{xy}

\begin{document}
	
	\title{Layers of the Coradical Filtration}
	
	\author[Alexander Sherman]{Alexander Sherman}

	\begin{abstract} Under suitably nice conditions, given a coalgebra object in a tensor category we compute the layers of its coradical (socle) filtration.  
\end{abstract}
	
	\maketitle
	\pagestyle{plain}


\section{Statement of Main Result}

Let $\Bbbk$ be a field and $\CC$ a semisimple pointed tensor category over $\Bbbk$ (precise definitions are given in \cref{sec_prelims}).  Recall that pointed means that every simple object of $\CC$ is invertible.  For instance $\CC$ could be the category of finite-dimensional (super) vector spaces.  Let $C$ be a coalgebra object in the cocompletion of $\CC$.  Then $C$ is a bicomodule over itself via its comultiplication morphism.  We prove a result on one aspect of this structure.

As a $C$-bicomodule, $C$ has an ascending Loewy series, i.e.~ its socle filtration:
\[
0=\sigma_0(C)\sub\sigma_1(C)\sub\sigma_2(C)\sub\cdots.
\]
This is often called the coradical filtration of $C$, and it is in fact a filtration of $C$ by coalgebras.  We seek to describe the layers of this filtration, as a bicomodule over $C$.  But we need a few assumptions on $C$.  

Before we state the assumptions, we recall a few constructions.  First, every simple right comodule $L$ of $C$ has an injective envelope $I(L)$, which is a right comodule that is an object of the cocompletion of $\CC$.  It too has a socle filtration $\sigma_{\bullet}(I(L))$ as a right comodule.  Next, given a right $C$-comodule $V$ and an object $S$ of $\CC$, the tensor product $S\otimes V$ has the natural structure of a right comodule.  Finally, if $V$ is a right $C$-comodule, then its left dual $V^*$ is a left $C$-comodule and if $W$ is a right comodule then the tensor product $V^*\otimes W$ has the natural structure of a bicomodule, which we denote by writing $V^*\boxtimes W$.  

Here are the assumptions we place on $C$: for the third assumption we fix $n\in\N$ with $n\geq 1$, meaning that this assumption depends on $n$.  It is possible that a given coalgebra $C$ only satisfies (C3$-n$) for certain $n$.  

\begin{enumerate}
	\item[\textbf{(C1)}] If $L$ is a simple right $C$-comodule and $S$ is a simple object of $\CC$, then $L$ and $S\otimes L$ are not isomorphic as comodules unless $S\cong\mathbf{1}$.  
	
	\item[\textbf{(C2)}] If $L,L'$ are simple right $C$-comodules then $L^*\boxtimes L'$ is a simple bicomodule and further every simple bicomodule is of this form up to isomorphism.
	
	\item[] \hspace{-2.6em}\textbf{(C3-}$n$\textbf{)} If $L,L'$ are simple right comodules, then, $[\sigma_n(I(L)):L']<\infty$.	
\end{enumerate}

Note that if $V$ and $L$ are right $C$-comodules of finite length, and $L$ is simple, we write $[V:L]$ number of times $L$ appears in a composition series of $V$. 

Finally, we observe that (C3-$n$) holds for all $n$ if for all simple right comodules $L,L'$ the following hold:
\begin{enumerate}
	\item[(a)] $\Ext^1(L,L')$ is finite-dimensional; and
	\item [(b)] for a fixed $L'$, $\Ext^1(L,L')$ vanishes for all but finitely many $L$ (up to isomorphism).
\end{enumerate} 
Indeed, in this case one can prove by induction that $\sigma_n(I(L))$ is of finite length for any simple module $L$: for $n=1$ it is clear, and we have the inequality:
\begin{eqnarray*}
	[\sigma_{n+1}(I(L))/\sigma_n(I(L)):L'']& \leq &\dim\Ext^1(L'',\sigma_n(I(L)))\\
	& \leq &\sum\limits_{[\sigma_n(I(L)):L']\neq 0}\dim\Ext^1(L'',(L')^{\oplus[\sigma_n(I(L)):L']}).
\end{eqnarray*}
The first inequality is clear, and the second equality follows from the fact that $\dim\Ext^1(L'',Z)\leq \dim\Ext^1(L'',X)+\dim\Ext^1(L'',Y)$ whenever we have a short exact sequence $0\to X\to Z\to Y\to 0$.  By our assumptions (a) and (b), the RHS is finite.  Thus
\[
[\sigma_{n+1}(I(L)):L'']\leq [\sigma_{n+1}(I(L))/\sigma_n(I(L)):L'']+[\sigma_n(I(L)):L'']<\infty.
\]
Our inequalities further show that $[\sigma_{n+1}(I(L)):L'']\neq0$ only if there exists simple right comodules $L_1,\dots,L_{n-1}$ such that $\Ext^1(L_1,L)\neq0,\Ext^1(L_2,L_1)\neq0,\dots,\Ext^1(L'',L_n)\neq0$.  By assumption (b), only finitely many simple right comodules satisfy this property. Thus $\sigma_{n+1}(I(L))$ will be of finite length.

\subsection{Main result} The main theorem we prove is:

\begin{thm} Assuming (C1)-(C3-$n$), if $i\leq n$ then for simple right comodules $L,L'$ we have
	\[
	[\sigma_i(C)/\sigma_{i-1}(C):L^*\boxtimes L']=[\sigma_i(I(L))/\sigma_{i-1}(I(L)):L'].
	\]
	In particular if (C3-$n$) holds for all $n$, then the above equality holds for all $i$.
\end{thm}

The following statement is clearly equivalent.

\begin{thm}\label{thm_2}
	Assuming (C1)-(C3-$n$), if $i\leq n$ then for simple right comodules $L,L'$ we have
	\[
	[\sigma_i(C):L^*\boxtimes L']=[\sigma_i(I(L)):L'].
	\]
	In particular if (C3-$n$) holds for all $n$, then the above equality holds for all $i$.
\end{thm}

In the case of $i=2$ we obtain a generalization of a corollary of the Taft-Wilson theorem for pointed coalgebras over a field.
\begin{cor} Assuming (C1)-(C2), if $L,L'$ are simple right comodules such that $\Ext^1(L,L')$ is finite-dimensional then we have
	\[
	[\sigma_2(C)/\sigma_1(C):L^*\boxtimes L']=\dim \operatorname{Ext}^1(L',L).
	\]
\end{cor}

The finiteness assumption in (C3-$n$) is clearly necessary in order to state the theorems.  The assumptions (C1) and (C2) are necessary for obtaining a clear description of the simple bicomodules of $C$.  If $A$ is a simple finite-dimensional $G$-graded algebra over an algebraically closed field of characteristic zero for a group $G$, and the center of $A$ contains a non-scalar element, then the assumptions (C1) and (C2) will fail for $C=A^*$ the dual coalgebra of $A$, as an object of $G$-graded vector spaces.

\subsection{Applications} The conditions (C1)-(C3) hold for many examples.  The case of original motivation and interest to the author appears in Section 6 of \cite{She}.  Namely, let $G$ be a quasireductive supergroup, that is one for which $G_0$ is reductive, and suppose that it has an even Cartan subgroup.  Then $G\times G$ acts on $G$ by left and right multiplication, and this induces an action of $G\times G$ on $\C[G]$.  One seeks a nice description of the structure of this $G\times G$-module; this is the natural generalization of the Peter-Weyl Theorem to the super setting.  Note that the structure of $\C[G]$ as a $G$-module, where $G$ acts by left translation, was given in \cite{Ser}.

The above situation is exactly given by the setup of this paper, where $\CC$ is the category of finite-dimensional super vector spaces over $\C$ and $C$ is the coalgebra $\C[G]$ ($\C[G]$ is in fact a Hopf algebra).  In this case, from \cref{thm_2}, one obtains a beautiful description of the Loewy layers of $\C[G]$ viewed as a $G\times G$-module.

More generally, if $\CC$ is the category of finite-dimensional vector spaces, then (C1) automatically holds, and if $C$ is a coalgebra over an algebraically closed field $\Bbbk$, then (C1)-(C2) hold\footnote{Thank you to Nicolás Andruskiewitsch for explaining why this is true.}.  More generally, if $G$ is a group, $\Bbbk$ an algebraically closed field of characteristic zero or characteristic $p$ where $p$ is coprime to the order of each finite subgroup of $G$, and $\CC$ is the category of $G$-graded vector spaces over $\Bbbk$, then (C1) and (C2) become equivalent.  This follows as a corollary of the main results of \cite{BZS}, that a finite-dimensional $G$-graded simple algebra $B$ is a matrix algebra over $\Bbbk$ if and only if the center of $B$ is $\Bbbk$.

\subsection{Outline of paper} In Section 2 we state formal constructions related to coalgebras and comodules in tensor categories, with \cite{EGNO} being our main reference.  In Section 3 we state basic results about the matrix coefficient morphism.  Section 4 goes into the existence and structure of injective comodules, and Section 5 explains the structure of the coalgebra as a right comodule.  The statements and proofs of these results are known and go back to \cite{Ser}.  Finally Section 6 examines the structure of $C$ as a bicomodule, concluding with \cref{main_thm}.

\subsection{Acknowledgments}  The author would like to thank his advisor, Vera Serganova, for many helpful discussions.  This research was partially supported by NSF grant DMS-1701532.  

\section{Setup and Preliminaries}\label{sec_prelims}

\subsection{} We follow the definitions and terminology from \cite{EGNO}.  Let $\Bbbk$ be a field and $\CC$  
a semisimple pointed tensor category over $\Bbbk$.  In other words we assume: 
\begin{enumerate}
	\item $\CC$ is a locally finite semisimple $\Bbbk$-linear abelian category;
	\item $\CC$ is rigid monoidal such that $(-)\otimes(-)$ is a biexact bilinear bifunctor, and $\End(\mathbf{1})\cong \Bbbk$;
	\item every simple object of $\CC$ is invertible.
\end{enumerate}
By Thm.~ 2.11.5 of \cite{EGNO}, such categories are always isomorphic (as monoidal categories) to a category $\mathsf{vec}(G,\omega)$, the category of finite-dimensional $G$-graded vector spaces (where $G$ is a group) with associativity isomorphism determined by the 3-cocycle $\omega\in Z^3(G,\Bbbk^\times)$.  Note that we do not assume $(\CC,\otimes)$ is braided. 

\subsection{} For an object $V$ of $\CC$, we write $V^*$ for its left dual, $\text{ev}_V:V^*\otimes V\to\mathbf{1}$ for the evaluation morphism and $\text{coev}_V:\mathbf{1}\to V\otimes V^*$ for the coevaluation morphism.  If $W$ is a subobject of $V$, we write $W^\perp$ for the subobject of $V^*$ given by the kernel of the epimorphism $V^*\to W^*$.  If $f:W\to V$ is an arbitrary morphism then we have a commutative diagram which will be used later on:
\begin{equation}\label{square}
\xymatrix{W^*\otimes W \ar[r]^{\text{ev}_W} & \mathbf{1}\\ V^*\otimes W \ar[u]^{f^*\otimes 1}\ar[r]_{1\otimes f} & V^*\otimes V\ar[u]^{\text{ev}_V}} 
\end{equation}

\subsection{} We consider the cocomplete abelian category $\hat{\CC}$ constructed from $\CC$, as described in \cite{Sta}.  Note that here if $\CC\cong\mathsf{vec}(G,\omega)$, then $\hat{\CC}\cong\mathbf{Vec}(G,\omega)$ which is the category of $G$-graded vector spaces of arbitrary dimension.   We have a fully faithful embedding $\CC\to\hat{\CC}$ admitting the usual universal property.  Further, in this case $\hat{\CC}\times\hat{\CC}$ is a cocomplete abelian category with a natural fully faithful functor $\CC\times\CC\to\hat{\CC}\times\hat{\CC}$ that satisfies the desired universal property.    Thus in particular $\otimes$ extends to a biexact bilinear functor $\hat{\CC}\times\hat{\CC}\to\hat{\CC}$ which we continue to write as $\otimes$ by abuse of notation.

\subsection{} Let $C$ be a coalgebra object in $\hat{\CC}$.  This means $C$ comes equipped with morphisms $\Delta:C\to C\otimes C$ and $\epsilon:C\to \mathbf{1}$ such that
\[
(\Delta\otimes \id_{C})\circ\Delta=(\id_{C}\otimes\Delta)\circ\Delta, \ \ \ \ \ (\epsilon\otimes \id_{C})\circ\Delta=(\id_{C}\otimes\epsilon)\circ\Delta=\id_{C}.
\]
By thinking of $\hat{\CC}$ as $\mathbf{Vec}(G,\omega)$, a standard argument shows that $\hat{\CC}$ is a direct limit of subcoalgebras objects of $\CC$.

\subsection{} An object $V\in\hat{\CC}$ is said to be a right $C$-comodule (resp. left $C$-comodule) if it is equipped with a morphism $a_V=a:V\to V\otimes C$ (resp. $a_V=a:V\to C\otimes V$) such that 
\[
(a\otimes\id_C)\circ a=(\id_V\otimes\Delta)\circ a, \ \ \ \ \ (\text{resp. } \ (\id_C\otimes a)\circ a=(\Delta\otimes\id_V)\circ a )  
\]
and
\[
(\id_V\otimes\epsilon)\circ a=\id_V, \ \ \ \ \ (\text{resp. } \ (\epsilon\otimes \id_V)\circ a=\id_V).
\]
An object $V\in\hat{\CC}$ is a $C$-bicomodule if it is both a left and right comodule with comodule structure morphisms $a_{V,l}$ and $a_{V,r}$ such that $(\id_C\otimes a_{V,r})\circ a_{V,l}=(a_{V,l}\otimes\id_C)\circ a_{V,r}$.  Observe that $C$ is naturally a left and right comodule via $a_{C,r}=a_{C,l}=\Delta$, such that it obtains the structure of a $C$-bicomodule.

Again by a standard argument, any $C$-(bi)comodule $V$ will be a sum of sub-(bi)comodule objects in $\CC$.  In particular, simple (bi)comodules are always objects of $\CC$.

\subsection{} Consider the category $\mathbf{Mod}_C$ (resp. $_C\mathbf{Mod}$) of right $C$-comodules (resp. left $C$-comodules) with morphisms between two objects $V,W$ being morphisms in $\hat{\CC}$ respecting comodule structure morphisms.  Let $_C\mathsf{mod}$ (resp. $\mathsf{mod}_C$) denote the full subcategory of right $C$-comodules (resp. left $C$-comodules) in $\CC$.  We also have the categories $_C\mathbf{Mod}_C$ and $_C\mathsf{mod}_C$ of $C$-bicomodules in $\hat{\CC}$ and $\CC$ respectively. By our assumption that $\otimes$ is biexact, these categories are all abelian.  Further, $_C\mathsf{mod}$, $\mathsf{mod}_C$ and $_C\mathsf{mod}_C$ are locally finite, and thus the Jordan-Holder and Krull-Schmidt theorems are valid.  The categories $_C\mathbf{Mod}$, $\mathbf{Mod}_C$ and $_C\mathbf{Mod}_C$ are cocomplete, and we have natural inclusion functors $\mathsf{mod}_C\to \mathbf{Mod}_C$, $_C\mathsf{mod}\to {_C\mathbf{Mod}}$, and $_C\mathsf{mod}_C\to {_C\mathbf{Mod}}_C$ that have the usual universal properties as cocompletions.

\subsection{} \label{tensor_product_structure_trivial} Given a right (resp. left) $C$-comodule $V$ and an object $S\in\hat{\CC}$, we may construct a new right (resp. left) $C$-comodule $S\otimes V$ (resp. $V\otimes S$) with comodule morphism $a_{S\otimes V}=\id_S\otimes a_V$ (resp. $a_{V\otimes S}=a_V\otimes \id_S$).  This defines an endofunctor of the categories $\mathbf{Mod}_C$ and $_C\mathbf{Mod}$, and it preserves $\mathsf{mod}_C$ and $_C\mathsf{mod}$ if $S$ is in $\CC$.  We observe that if $S$ is simple (and thus invertible) then this functor defines automorphisms of these abelian categories, and thus it takes simple comodules to simple comodules.

\subsection{} \label{bicomodule_construction} Given a right $C$-comodule $V$ and left $C$-comodule $W$ we may construct a $C$-bicomodule $V\boxtimes W$, which is $V\otimes W$ as an object of $\hat{\CC}$ and has left and right comodule structures as described in \ref{tensor_product_structure_trivial}.  This satisfies the necessary commutativity condition to be a bicomodule. 

\begin{lemma}\label{associativity_bicom_iso}
	Suppose that $V$ is a right $C$-comodule, $W$ a left $C$-comodule, and $S$ is an object of $\hat{\CC}$.  Then we have a canonical isomorphism of bicomodules
	\[
	(W\otimes S)\boxtimes V\cong W\boxtimes(S\otimes V).
	\]
\end{lemma}

\begin{proof}
	Indeed, the associativity isomorphism coming from the monoidal structure of $\hat{\CC}$ provides us with such an isomorphism.
\end{proof}
\begin{cor}\label{cancellation_bicom_iso}
With the same hypotheses as \cref{associativity_bicom_iso} and assuming that $S$ is a simple object of $\CC$, we have a canonical isomorphism
\[
(W\otimes S^*)\boxtimes (S\otimes V)\cong W\boxtimes V.
\]
\end{cor}
\begin{proof}
	Apply \cref{associativity_bicom_iso} and the invertibility isomorphism $S^*\otimes S\cong\mathbf{1}$.  
\end{proof}

\subsection{} \label{dual_construction}  Let $V$ be an object in $\mathsf{mod}_C$ and $V^*$ its left dual in $\CC$.  Then $V^*$ has the natural structure of a left $C$-comodule by 
\[
a_{V^*}=(\operatorname{ev}_V\otimes \id_C\otimes\id_{V^*})\circ (\id_{V^*}\otimes a_V\otimes \id_{V^*})\circ(\id_{V^*}\otimes\operatorname{coev}_V).
\]
This construction is functorial, so that we have a contravariant functor $(-)^*:\mathsf{mod}_C\to{_C\mathsf{mod}}$.  This functor is an antiequivalence with inverse taking the right dual of a comodule, $V\mapsto{^*V}$.  We observe that if $W$ is a right subcomodule of $V$ then $W^\perp$ is naturally a left subcomodule of $V^*$.  

\section{Matrix Coefficients}

\subsection{} For this section, all objects are assumed to be in $\CC$, i.e.~ they are of finite length.  Given an object $V$ of $\mathsf{mod}_C$, by \ref{dual_construction} and \ref{bicomodule_construction} we obtain a $C$-bicomodule given by $V^*\boxtimes V$.  Define the matrix coefficients morphism $c_V:V^*\boxtimes V\to C$ by
\[
c_V=(\operatorname{ev}_V\otimes \id_{C})\circ (\id_{V^*}\otimes a_V)=(\id_C\otimes\operatorname{ev}_V)\circ (a_{V^*}\otimes\id_V).
\]
\begin{lemma}\label{matrix_coeff_square}
	Suppose that $f:W\to V$ is a morphism of right $C$-comodules.  Then we have the following commutative diagram:
	\[
	\xymatrix{ V^*\boxtimes V \ar[r]^{c_V} & C\\ V^*\boxtimes W \ar[u]^{f^*\otimes 1}\ar[r]_{1\otimes f} & W^*\boxtimes W\ar[u]^{c_W}}
	\]
\end{lemma}
\begin{proof}
	Indeed, this follows from the commutativity of the following diagram:
	\[
		\xymatrix{W^*\otimes W \ar[rr]^{\id_{W^*}\otimes a_W}& & W^*\otimes W\otimes C \ar[rr]^{\text{ev}_W\otimes\id_C} && C \\ \\
			& & V^*\otimes W\otimes C \ar[uu]\ar[rr]  & & V^*\otimes V\otimes C\ar[uu]_{\text{ev}_V\otimes\id_C}\\ \\
		V^*\otimes W\ar[rrrr]\ar[uuuu]\ar[uurr]^{\id_{V^*}\otimes a_W} & && & V^*\otimes V\ar[uu]_{\id_{V^*}\otimes a_V}}
	\]
	The top left square is obviously commutative.  The bottom right square is commutative because $f:W\to V$ is a morphism of comodules.  The top right square is simply (\ref{square}) tensored with $C$, and thus is commutative.  
\end{proof}
\begin{cor}\label{ker_of_c}
	Suppose that $V$ is a right $C$-comodule, with $W$ a sub-comodule of $V$.  Then $W^\perp\boxtimes W\sub\ker c_V$.
\end{cor}
\begin{proof}
	Clear from previous lemma.
\end{proof}

\begin{lemma}\label{inj_surj}
	Suppose that $V$ is a right $C$-comodule, with $W$ a sub-comodule of $V$ and $U$ a quotient of $V$.  Then $\im c_W$ and $\im c_U$ are sub-bicomodules of $\im c_V$.  
\end{lemma}

\begin{proof}
	Apply the commutative squares obtained from Lemma \ref{matrix_coeff_square}	for the two morphisms $W\to V$ and $V\to U$.  
\end{proof}
\begin{cor}
If $V$ is a right $C$-comodule and $W$ is a subquotient of $V$ as a comodule, then $\im c_W$ is a sub-bicomodule of $\im c_V$.  
\end{cor}

\begin{lemma}
	Let $V$ be a right $C$-comodule, and suppose that $W_1,W_2$ are subcomodules such that $W_1+W_2=V$.  Then $\im c_V=\im c_{W_1}+\im c_{W_2}$. Similarly, if $U_1,U_2$ are quotients comodules of $V$ such that the map $V\to U_1\oplus U_2$ is injective, then $\im c_V=\im c_{U_1}+\im c_{U_2}$.  
\end{lemma}

\begin{proof}
	We apply \cref{inj_surj} to the epimorphism $W_1\oplus W_2\to V$ and monomorphism $V\to U_1\oplus U_2$, and use Corollary \ref{ker_of_c} to find that $c_{W_1\oplus W_2}$ (resp. $c_{U_1\oplus U_2}$) factors through $c_{W_1}\oplus c_{W_2}$ (resp. $c_{U_1}\oplus c_{U_2}$).  
\end{proof}

\subsection{} Given a finite-length right $C$-subcomodule $V$ of $C$, let $\epsilon_V:V\to \mathbf{1}$ be the restriction of $\epsilon$ to $V$ and $\epsilon_V^*:\mathbf{1}\to V^*$ its dual.  Then the following is a commutative diagram of right $C$-comodules:
\[
\xymatrix{V\ar[rr]^{\epsilon_V^*\otimes 1}\ar[drr] & & V^*\otimes V\ar[d]_{c_V}\\ && C}
\]

Thus $V$ is a right subcomodule of the image of $c_V$ in $C$.  Since $C$ is the sum of its finite length right sub-comodules, it follows that $C=\sum\im c_V$, where the sum runs over all right $C$-comodules in $\CC$.  

\section{Socle Filtration and Injectives}

\subsection{} The objects of $\mathbf{Mod}_C$, $_C\mathbf{Mod}$, and $_C\mathbf{Mod}_C$ admit socle filtrations.  Using the same notation as Green in \cite{Gre}, we write $\sigma_i(V)$ for the $i$th term in the socle filtration of an object $V$.  In this case we have that $V$ is the direct limit of its socle filtration. If the socle filtration of an object $V$ is finite (which happens in particular if $V$ is of finite-length, i.e.~ is in $\CC$), then we write $\ell\ell(V)$ for the length of the socle filtration, the Loewy length of $V$.  In this case, $\ell\ell(V)$ is the length of every minimal semisimple filtration of $V$.  Further, then $V$ also has a radical filtration which is a descending filtration whose $i$th term we write as $\rho^i(V)$, and whose length is also $\ell\ell(V)$.  Recall that $\rho^1(V):=\rho(V)$ is defined to be the minimal subcomodule of $V$ such that $V/\rho(V)$ is semisimple, and we define the filtration inductively by $\rho^i(V)=\rho(\rho^{i-1}(V))$.

\begin{lemma}
	If $V$ is of finite length, then $\sigma_i(V)^\perp=\rho^i(V^*)$ and $\rho^i(V)^\perp=\sigma_i(V^*)$.  
\end{lemma}
\begin{proof}
	Follows from the fact that dualizing is an antiequivalence of comodule categories.
\end{proof}
\subsection{} The socle filtration on $C$ as a $C$-bicomodule is often called the coradical filtration of $C$, and is sometimes written $C_i:=\sigma_i(C)$.  The goal of this paper is to give a description of the layers of the coradical filtration of $C$.

\subsection{} Define the functor $F_C:\hat{\CC}\to{\mathbf{Mod}_C}$ by $F_C(S)=S\otimes C$ (see \ref{tensor_product_structure_trivial}).

\begin{lemma}\label{adjoint_functor}
	The functor $F_C$ is right adjoint to the forgetful functor $\mathbf{Mod}_C\to\hat{\CC}$.
\end{lemma}
\begin{proof}
The proof follows the same ideas as in (1.5a) of \cite{Gre}.
\end{proof}

\subsection{} \begin{lemma}
	The categories $_C\mathbf{Mod}$ and $\mathbf{Mod}_C$ have enough injectives.
\end{lemma}

\begin{proof}
Given a right $C$-comodule $V$, $F_C(V)$ is injective by \cref{adjoint_functor} and the morphism $a_V:V\to F_C(V)$ is a monomorphism of right $C$-comodules. 
\end{proof}

\begin{lemma}
	The direct sum of injective comodules is injective.
\end{lemma}

\begin{proof}
The proof in (1.5b) of \cite{Gre} carries through to our case.
\end{proof}

Given a right $C$-comodule $V$, an injective envelope of $V$ is the data of an injective right comodule $I$ with a monomorphism $V\to I$ which induces an isomorphism $\sigma(V)\xto{\sim}\sigma(I)$.  An injective envelope is unique up to isomorphism if it exists.  Using Brauer's idempotent lifting process as described in \cite{Gre}, we can prove that injective envelopes always exists.  Choose for each simple right comodule $L$ an injective envelope $I(L)$.  We now have: 

\begin{cor}
	The indecomposable injective right comodules are exactly those of the form $I(L)$ for a simple right comodule $L$.  Thus the injective right comodules are exactly the direct sums of injective indecomposables $I(L)$.  
\end{cor}

\section{Structure of $C$ as a right comodule}
We now make some assumptions on $C$ and its right comodule category.

\begin{enumerate}
	\item[\textbf{(C1)}] We suppose that if $L$ is a simple right $C$-comodule and $S$ is a simple object of $\CC$, then $L$ and $S\otimes L$ are not isomorphic as comodules unless $S\cong\mathbf{1}$.  
	
	\item[\textbf{(C2)}] We assume that if $L,L'$ are right $C$-comodules then $L^*\boxtimes L'$ is a simple bicomodule, and further every simple bicomodule is of this form.
\end{enumerate}
\begin{remark}\label{C3}
	Assumption (C2) implies that every semisimple bicomodule is semisimple as a right comodule.  In particular, if $V$ is a bicomodule of finite length then its Loewy length as a right comodule is less than or equal to its Loewy length as a bicomodule.
\end{remark}

Assumption (C1) is saying that the action of the Picard group of $\CC$ (the group of invertible objects of $\CC$ modulo isomorphism, under tensor product) on the set of simple comodules is free.  Thus we may, and do, choose representatives of each orbit, $\{L_{\alpha}\}_{\alpha}$.  In other words the simple right comodules $L_{\alpha}$ have the property that if $L_{\alpha}\cong S\otimes L_{\beta}$ for a simple object $S$ of $\CC$, then $\alpha=\beta$ and $S\cong \mathbf{1}$.  Further if $L$ is a simple right comodule then there exists an $\alpha$ and a simple object $S$ of $\CC$ such that $L\cong S\otimes L_{\alpha}$.

\begin{lemma}\label{uniqueness_bicomodules}
	Every simple bicomodule may be written as $L_{\alpha}^*\boxtimes L$ for a unique $\alpha$ and a unique simple right comodule $L$, up to isomorphism.   
\end{lemma}

\begin{proof}
	By (C2) the simple bicomodules are all of the form $(L')^*\boxtimes L''$ for some simple right comodules $L',L''$.  Choose $\alpha$ such that $L'\cong S\otimes L_{\alpha}$.  Then by \cref{associativity_bicom_iso} $(L')^*\boxtimes L''\cong (L_{\alpha}^*\otimes S^*)\boxtimes L''\cong L_{\alpha}^*\boxtimes(S^*\otimes L'')$.  Setting $L=S^*\otimes L''$ we have proven the first half of our claim.
	
	The proof of uniqueness of $L$ is a little trickier.  We prove the following statement which implies it: if $L$ is a simple right comodule, $L'$ a simple left comodule, and $S$ is a simple object of $\CC$, then if $(L'\otimes S)\boxtimes L\cong L'\boxtimes L$ then $S\cong\mathbf{1}$.  Write $G=\operatorname{Pic}(\CC)$ for the Picard group of $\CC$, that is the group of simple (hence invertible) objects of $\CC$ up to isomorphism under tensor product.  For each $g\in G$, choose a representative simple object $S_g$,  and let $h\in G$ be the class of $S$ so that $S\cong S_h$.  Finally, write $\phi:(L'\otimes S)\boxtimes L\to L'\boxtimes L$ for a given isomorphism of bicomodules.
	
	Now write as objects of $\CC$ isotypic decompositions $L'=\bigoplus\limits_g T_g$, where $T_g\cong S_g^{\oplus n_g}$ and $L=\bigoplus\limits_g U_g$ where $U_g\cong S_g^{\oplus m_g}$. The isomorphism $\phi$ of bicomodules gives rise to an isomorphism of right comodules
	\[
	\bigoplus\limits_{g}(T_g\otimes S)\otimes L\cong\bigoplus\limits_{g}T_{g}\otimes L
	\]
	and an isomorphism of left comodules
	\[
 	\bigoplus\limits_{g} L'\otimes (S\otimes U_{g})\cong \bigoplus\limits_g L'\otimes U_{g}.
	\]
	By (C1), this must induce isomorphisms of right comodules
	\begin{equation}\label{eqn1}
	(T_g\otimes S)\otimes L\cong T_{gh}\otimes L,
	\end{equation}
	i.e.~ $\phi$ must take $(T_g\otimes S)\otimes L$ into $T_{gh}\otimes L$ for all $g\in G$, and similarly of left comodules
	\begin{equation}\label{eqn2}
	L'\otimes (S\otimes U_g)\cong L'\otimes U_{hg},
	\end{equation}
	i.e.~ $\phi$ must take $L'\otimes (S\otimes U_g)$ into $L'\otimes U_{hg}$ for all $g\in G$.	However for $g,h,k\in G$, (\ref{eqn1}) implies that $\phi$ induces an isomorphism
	\[
	T_g\otimes S\otimes U_k\cong T_{gh}\otimes U_k
	\]
	while (\ref{eqn2}) implies that $\phi$ induces an isomorphism
	\[
	T_{gh}\otimes S\otimes U_{h^{-1}k}\cong T_{gh}\otimes U_k.
	\]
	Thus we learn that $\phi$ takes both $T_g\otimes S\otimes U_k$ and $T_{gh}\otimes S\otimes U_{h^{-1}k}$ isomorphically to $T_{gh}\otimes U_k$; but since these are objects of $\CC$, and hence of finite length, this forces an equality of subobjects of $L'\boxtimes L$, namely that $T_g\otimes S\otimes U_k=T_{gh}\otimes S\otimes U_{h^{-1}k}$.  But these subobjects are distinct unless $gh=g$ and $h^{-1}k=k$ i.e.~ $h$ must be the identity, and so $S\cong \mathbf{1}$ as desired.
	
\end{proof}

\begin{lemma}\label{im_c}
	If $S$ is a simple object of $\CC$ and $V$ is in $\mathsf{mod}_C$, then $\im c_{V}=\im c_{S\otimes V}$.  In particular if $L$ is a simple right comodule and $L\cong S\otimes L_{\alpha}$, then $\im c_{L}=\im c_{L_\alpha}$.
\end{lemma}
\begin{proof}
	By \cref{cancellation_bicom_iso} we have $V^*\boxtimes V\cong (S\otimes V)^*\boxtimes (S\otimes V)$, and this isomorphism of bicomodules respects the matrix coefficient morphisms.
\end{proof}

\begin{prop}\label{socle_of_C}
	We have $\sigma(C):=\sigma_1(C)=\bigoplus\limits_{\alpha}L_{\alpha}^*\boxtimes L_{\alpha}$ as bicomodules.
\end{prop}

\begin{proof}
	For each $\alpha$ we have a nonzero, and thus injective, morphism $c_{L_{\alpha}}:L_{\alpha}^*\boxtimes L_{\alpha}\to \sigma(C)$.  Since the simple bicomodules $L_{\alpha}^*\boxtimes L_{\alpha}$, $L_{\beta}^*\boxtimes L_{\beta}$ are non-isomorphic for distinct $\alpha,\beta$ by \cref{uniqueness_bicomodules}, we obtain an inclusion $\bigoplus\limits_{\alpha}L_{\alpha}^*\boxtimes L_{\alpha}\sub \sigma(C)$.  Conversely, a simple sub-bicomodule $W$ of $C$ must be semisimple as a right $C$-comodule by \cref{C3}, and thus if $W=\bigoplus\limits_i L_i$ for simple right comodules $L_i$ then $W\sub\sum\limits_i\im c_{L_i}$.  Now \cref{im_c} completes the proof.
\end{proof}

\begin{cor}
	We have an isomorphism of right comodules:
	\[
	C\cong\bigoplus_{\alpha} L_{\alpha}^*\otimes I(L_{\alpha})
	\]
\end{cor}
\begin{proof}
	By \cref{socle_of_C} these right comodules have isomorphic socles.  Since injectives are determined by their socles, we are done.
\end{proof}

\section{Layers of the coradical filtration}

\subsection{}  We would like to prove that if $V$ is a finite-length right $C$-comodule then $\ell\ell(\im c_{V})=\ell\ell(V)$, i.e.~ the Loewy length of $\im c_V$ as bicomodule is equal to the Loewy length of $V$ as a right comodule. First we prove a lemma.

\begin{lemma}\label{injectivity_res_lemma}
	Suppose that $W$ is a right comodule of finite length with simple socle $L$.  Choose a splitting $\tilde{L}\sub W^*$  (in $\CC$) of the epimorphism $W^*\to L^*$ so that we obtain a right subcomodule $\tilde{L}\otimes W$ of $W^*\boxtimes W$.  Then the restriction of $c_W$ to $\tilde{L}\otimes W$ is injective.
\end{lemma}
\begin{proof}
	Since this restriction defines a morphism of right comodules $\tilde{L}\otimes W\to C$, it suffices to show that it is injective on the socle $\sigma(\tilde{L}\otimes W)=\tilde{L}\otimes L$.  However $\tilde{L}\otimes L$ is a splitting of the head of the bicomodule $W^*\boxtimes L$, and the restriction of $c_W$ to $W^*\boxtimes L$ has $L^\perp\boxtimes L=\rho(W^*\boxtimes L)$ in its kernel by \cref{ker_of_c} , and thus factors through $(W^*\boxtimes L)/(L^\perp\boxtimes L)\cong L^*\boxtimes L\xto{c_L}C$.  In summary we have a commutative diagram
	\[
	\xymatrix{\tilde{L}\otimes L\ar[r]\ar[dr]_{c_W|_{\tilde{L}\otimes L}} & W^*\otimes L\ar[r] & L^*\boxtimes L\ar[dl]^{c_L}\\
	& C &}
	\]
	Since the composition $\tilde{L}\otimes L\to L^*\boxtimes L$ is an isomorphism and $c_L$ is injective we are done.
\end{proof}

\begin{lemma}
Let $V$ be a finite-length right comodule with $\ell\ell(V)=n$.  Then 
\[
F_k=\sum\limits_{i+j=k} \rho^{n-i}(V^*)\boxtimes\sigma_{j}(V)=\sum\limits_{i+j=k} \sigma_{n-i}(V)^\perp\boxtimes\sigma_{j}(V)
\] 
is a semisimple filtration of $V^*\boxtimes V$ such that $F_1=0$ and $F_{2n}=V^*\boxtimes V$.  
\end{lemma}

\begin{proof}
	The tensor product of semisimple filtrations is again a semisimple filtration.
\end{proof}

We now observe that $F_n=\sum\limits_{i}\sigma_i(V)^\perp\boxtimes\sigma_i(V)\sub\ker c_V$ and thus $F_\bullet$ induces a semisimple filtration of $\im c_V$ of length at most $n=\ell\ell(V)$.  It follows that $\ell\ell(\im c_V)\leq\ell\ell(V)$.  

On the other hand $V$ contains a subquotient $W$ with $\ell\ell(W)=\ell\ell(V)$ such that $W$ has a simple socle.  Since $\im c_W\sub\im c_V$, if we can show that $\ell\ell(\im c_W)\geq \ell\ell(W)=\ell\ell(V)$ then we will have that $\ell\ell(\im c_V)=\ell\ell(V)$.  

By \cref{injectivity_res_lemma} we know that $\im c_W$ contains a right subcomodule of the form $\tilde{L}\otimes W$ for an object $\tilde{L}$ of $\CC$, and thus its Loewy length as a right comodule is at least $\ell\ell(W)$, which by \cref{C3} implies its Loewy length as a bicomodule is at least $\ell\ell(W)$.  We have now finished showing:
\begin{prop}\label{Loewy_length_c_V}
	For a finite length right comodule $V$, $\ell\ell(V)=\ell\ell(\im c_V)$.
\end{prop}

\begin{prop}\label{socle_filtration_C_general}
	We have
	\[
	\sigma_i(C)=\sum\limits_{\ell\ell(V)\leq i}\im c_{V}.
	\]
\end{prop}
\begin{proof}
	By \cref{Loewy_length_c_V}, $\im c_V$ has Loewy length equal to that of $V$, so if $\ell\ell(V)\leq i$ then $\im c_V=\sigma_i(\im c_V)\sub\sigma_i(C)$.  Conversely if $V\sub \sigma_{i}(C)$ is a right sub-comodule then by \cref{C3} $\ell\ell(V)\leq i$ and so $V\sub\im c_{V}\sub\sigma_i(C)$.  Since $\sigma_i(C)$ is the sum of its right subcomodules, we are done.
\end{proof}

\subsection{}  From now on, fix $n\in\N$ with $n\geq 1$.  We make a finiteness assumption on the comodule category $\mathsf{mod}_C$.
\begin{enumerate}
	\item[\textbf{(C3-}$n$\textbf{)}] Assume that if $L,L'$ are simple right comodules, then $[\sigma_n(I(L)):L']<\infty$.	
\end{enumerate}
Note that (C3-$n$) implies (C3-$m$) whenever $m\leq n$.

\subsection{} For each pair of simple right comodules $L,L'$ and for each $i\leq n$ we define $H_{L',L}^i$ to be the right subcomodule of $\sigma_i(I(L'))$ that is generated by a splitting of the isotypic component of $L$ in $\sigma_i(I(L'))/\sigma_{i-1}(I(L'))$.  In particular it is zero if and only if $[\sigma_i(I(L'))/\sigma_{i-1}(I(L')):L]=0$.  By (C3-$n$), $H_{L',L}^i$ is a finite length right comodule.  Further we have for $i\leq n$
\begin{equation}\label{socle_in_sum}
\sigma_i(I(L'))=\sum\limits_{L\text{ simple}}\sum\limits_{j\leq i}H_{L',L}^j.
\end{equation}
Write $H_{\alpha,L}^i:=H_{L_{\alpha},L}^i$.  

\begin{lemma}\label{socle_filtration_specific}  For $i\leq n$,
	\[
	\sigma_i(C)=\sum\limits_{\alpha}\sum\limits_{L \operatorname{ simple}}\sum\limits_{j\leq i}\im c_{H_{\alpha,L}^j}
	\]
\end{lemma}
\begin{proof}
	Since $\ell\ell(H_{\alpha,L}^j)\leq j\leq i$, by \cref{socle_filtration_C_general} it suffices to show that $\im c_V$ is contained in the RHS whenever $V$ is a right comodule of Loewy length less than or equal to $i$.  However in this case $\im c_V=\sum\limits_{W}\im c_W$ where the sum runs over quotients of $V$ with simple socles.  Note that $\ell\ell(W)\leq \ell\ell(V)\leq i$ for all such $W$.  On the other hand, if $W$ has a simple socle $L'$ then after potentially twisting $W$ by a simple object $S$ (which won't change $\im c_{W}$) we may assume $L'\cong L_{\alpha}$ for some $\alpha$, and then $I(L_{\alpha})$ is the injective envelope of $W$.  If $\ell\ell(W)\leq i$ then $W\sub\sigma_i(I(L_{\alpha}))$ under an embedding of $W$ in $I(L_{\alpha})$.  Therefore by \ref{socle_in_sum},
	\[
	W\sub \sum\limits_{L\text{ simple}}\sum\limits_{j\leq i}H_{\alpha,L}^j.
	\] 
	and so there exists finitely many simple right comodules $L_1,\dots,L_n$ such that  
	\[
	W\sub \sum\limits_{k,\ j\leq i}H_{\alpha,L_k}^j
	\]
	and hence
	\[
	\im c_W\sub \sum\limits_{k,\ j\leq i}\im c_{H_{\alpha,L_k}^j}.
	\]
\end{proof}

\subsection{} We may now state the main theorem.

\begin{thm}\label{main_thm}
   For $i\leq n$, 
	\[
	[\sigma_i(C)/\sigma_{i-1}(C):L^*\boxtimes L']=[\sigma_i(I(L))/\sigma_{i-1}(I(L)):L'].
	\]
\end{thm}

\begin{proof}
    The case of $i=1$ is \cref{socle_of_C}.  If $n=1$ then the theorem is proven.
	
	Otherwise if $n>1$ we consider the case $i>1$.  We use \cref{socle_filtration_specific} and study the contribution of $\im c_{H_{\alpha,L}^i}$ for a fixed simple comodule $L$.  Write $V_{i}=H_{\alpha,L}^i$ and $V_{i-1}=\sigma_{i-1}(H_{\alpha,L}^{i})$ so that $V_{i}/V_{i-1}$ is a sum of copies of $L$.  Consider the sub-bicomodule $W=V_i^*\boxtimes V_{i-1}+(V_i/L_{\alpha})^*\boxtimes V_i$ of $V_i^*\boxtimes V_i$.  We see from the arguments of \cref{matrix_coeff_square} that 
	\[
	c_{V_i}(W)\sub \im c_{V_{i-1}}+\im c_{V_i/L_{\alpha}},
	\]
	and since $\ell\ell(V_{i-1}),\ell\ell(V_i/L_{\alpha})\leq i-1$, we find that $c_{V_i}(W)\sub\sigma_{i-1}(C)$.  We have
	\[
	(V_i^*\boxtimes V_i)/W\cong L_{\alpha}^*\boxtimes V_i/V_{i-1}
	\] 
	and so we have epimorphisms
	\[
	 L_{\alpha}^*\boxtimes V_i/V_{i-1}\cong V_i^*\boxtimes V_i/W\to \im c_{V_i}/c_{V_i}(W)\to \im c_{V_i}/(\sigma_{i-1}(C)\cap \im c_{V_i}).  \ \ \ \ \     (*)
    \]
	We aim to show this composition (*) is in fact an isomorphism.  To this end, choose a splitting $\tilde{L_{\alpha}}$ of $V_i^*\to L_{\alpha}^*$ so that we get a right subcomodule $\tilde{L_{\alpha}}\otimes V_i$ of $V_i^*\otimes V_i$.  By \cref{injectivity_res_lemma}, the restriction of $c_{V_i}$ to $\tilde{L_{\alpha}}\otimes V_i$ will be injective.  Further, as a right comodule we have 
	\[
	\sigma_{i-1}(\tilde{L_{\alpha}}\otimes V_i)=\tilde{L_{\alpha}}\otimes V_{i-1}.
	\] 
	Thus by \cref{C3}
	\[
	\sigma_{i-1}(C)\cap c_{V_i}(\tilde{L_{\alpha}}\otimes V_i)\sub c_{V_i}(\tilde{L_{\alpha}}\otimes V_{i-1}).
	\]
	Conversely $\tilde{L_{\alpha}}\otimes V_{i-1}\sub W$ and therefore
	\[
	c_{V_i}(\tilde{L_{\alpha}}\otimes V_{i-1})\sub \sigma_{i-1}(C)\cap c_{V_i}(\tilde{L_{\alpha}}\otimes V_i)
	\]
	which implies these are equal.  It follows that we obtain an injection of right comodules
	\[
	L_{\alpha}^*\otimes V_{i}/V_{i-1}\to \im c_{V_i}/(\sigma_{i-1}(C)\cap \im c_{V_i})
	\]
    and so (*) is an isomorphism.  What this shows is that the contribution of $\im c_{H_{\alpha,L}^i}$ to $\sigma_i(C)/\sigma_{i-1}(C)$ is exactly $L_{\alpha}^*\boxtimes V_i/V_{i-1}$.  By \cref{uniqueness_bicomodules} it follows that 
    \[
    \sigma_{i}(C)/\sigma_{i-1}(C)=\bigoplus\limits_{\alpha}L_{\alpha}^*\boxtimes \sigma_i(I(L_{\alpha}))/\sigma_{i-1}(I(L_{\alpha})).
    \]
Thus we have proven the theorem whenever $L\cong L_\alpha$ for some $\alpha$.  For the general case we write $L\cong S\otimes L_{\alpha}$ for some $\alpha$ and some simple object $S$ of $\CC$ and derive the result from \cref{associativity_bicom_iso}.
\end{proof}
We now obtain a generalization of the Taft-Wilson theorem for pointed coalgebras over a field.
\begin{cor}  Assume (C1)-(C2) and that $L,L'$ are simple right comodules such that $\dim \operatorname{Ext}^1(L,L')<\infty$.  Then
	\[
	[\sigma_2(C)/\sigma_1(C):L^*\boxtimes L']=\dim \operatorname{Ext}^1(L,L').
	\]
\end{cor}

\textsc{\footnotesize School of Mathematics and Statistics, University of Sydney, Australia} 

\textit{\footnotesize Email address:} \texttt{\footnotesize xandersherm@gmail.com}


\begin{thebibliography}{5}
	
	\bibitem[BZS]{BZS} Y.A.~ Bahturin, M.V.~ Zaicev, and S.K. Sehgal. \textit{Finite-dimensional simple
	graded algebras}, Sbornik: Mathematics, Vol.~199, no.~7 (2008):  965–983 .
	\bibitem[EGNO]{EGNO} P. Etingof, S. Gelaki, D. Nikshych, and V. Ostrik. \textit{Tensor categories}, American Mathematical Soc., Vol.~	205 (2016).
	\bibitem[Gre]{Gre} J.A.~ Green. \textit{Locally finite representations}, Journal of Algebra Vol.~41, no. 1 (1976): 137–171.
	\bibitem[Ser]{Ser}  V. Serganova. \textit{Quasireductive supergroups}, New developments in Lie theory and its applications, Vol.~544 (2011): 141–159.
	\bibitem[She]{She} A. Sherman. \textit{Spherical supervarieties}, Annales de l’Institut Fourier, Vol.~71, no.~4 (2021): 1449–1492.
	\bibitem[Sta]{Sta} H.B.~ Stauffer. \textit{The completion of an abelian category}, Transactions of the American Mathematical
	Society, Vol.~170 (1972): 403–414.
	
\end{thebibliography}
\end{document}